\documentclass{article}
\usepackage[T1]{fontenc}
\usepackage[latin9]{inputenc}
\usepackage[english,british]{babel}
\usepackage{mathrsfs}
\usepackage{mathtools}
\usepackage{amsmath}
\usepackage{amsthm}
\usepackage{amssymb}
\usepackage[authoryear]{natbib}
\usepackage[unicode=true,pdfusetitle,
 bookmarks=true,bookmarksnumbered=false,bookmarksopen=false,
 breaklinks=false,pdfborder={0 0 1},backref=false,colorlinks=false]
 {hyperref}

\makeatletter
\numberwithin{equation}{section}
\theoremstyle{plain}
\newtheorem{thm}{\protect\theoremname}
\theoremstyle{plain}
\newtheorem{assumption}[thm]{\protect\assumptionname}
\theoremstyle{definition}
\newtheorem{defn}[thm]{\protect\definitionname}
\theoremstyle{plain}
\newtheorem{cor}[thm]{\protect\corollaryname}
\theoremstyle{remark}
\newtheorem*{rem*}{\protect\remarkname}
\theoremstyle{plain}
\newtheorem{lem}[thm]{\protect\lemmaname}

\@ifundefined{date}{}{\date{}}
\usepackage[a4paper, includeheadfoot, margin = 2.54cm]{geometry} 
\usepackage{amsmath}
\usepackage[capitalize,nameinlink]{cleveref}
\crefname{equation}{}{}
\usepackage{autonum}

\usepackage{bbm}

\makeatletter
\newcommand{\restore@Environment}[1]{%
  \AtBeginDocument{%
    \csletcs{#1*}{#1}%
    \csletcs{end#1*}{end#1}%
  }%
}
\forcsvlist\restore@Environment{alignat,equation,gather,multline,flalign,align}
\makeatother

\AtBeginDocument{\renewcommand{\ref}[1]{\cref{#1}}}

\allowdisplaybreaks

\usepackage{thmtools}
\makeatletter
\def\cleartheorem#1{%
    \expandafter\let\csname#1\endcsname\relax
    \expandafter\let\csname c@#1\endcsname\relax
}
\makeatother

\cleartheorem{thm}
\cleartheorem{cor}
\cleartheorem{lem}
\cleartheorem{prop}
\cleartheorem{conjecture}
\cleartheorem{defn}
\cleartheorem{example}
\cleartheorem{problem}
\cleartheorem{xca}
\cleartheorem{sol}
\cleartheorem{rem}
\cleartheorem{claim}
\cleartheorem{fact}
\cleartheorem{criterion}
\cleartheorem{lyxalgorithm}
\cleartheorem{ax}
\cleartheorem{condition}
\cleartheorem{note}
\cleartheorem{notation}
\cleartheorem{summary}
\cleartheorem{acknowledgement}
\cleartheorem{conclusion}
\cleartheorem{assumption}
\cleartheorem{question}

\theoremstyle{plain}
\newtheorem{thm}{\protect\theoremname}
\theoremstyle{plain}
\newtheorem{cor}[thm]{\protect\corollaryname}
\theoremstyle{plain}
\newtheorem{lem}[thm]{\protect\lemmaname}
\theoremstyle{plain}

\theoremstyle{plain}

\theoremstyle{definition}
\newtheorem{defn}[thm]{\protect\definitionname}
\theoremstyle{definition}

\theoremstyle{definition}

\theoremstyle{definition}

\theoremstyle{definition}

\theoremstyle{remark}

\theoremstyle{remark}

\theoremstyle{plain}

\theoremstyle{plain}

\theoremstyle{plain}

\theoremstyle{plain}

\theoremstyle{definition}

\theoremstyle{remark}

\theoremstyle{remark}

\theoremstyle{remark}

\theoremstyle{remark}

\theoremstyle{remark}

\theoremstyle{plain}
\newtheorem{assumption}[thm]{\protect\assumptionname}
\theoremstyle{plain}

\crefname{thm}{Theorem}{Theorems}
\crefname{cor}{Corollary}{Corollaries}
\crefname{lem}{Lemma}{Lemmata}
\crefname{prop}{Proposition}{Propositions}
\crefname{conjecture}{Conjecture}{Conjectures}
\crefname{defn}{Definition}{Definitions}
\crefname{example}{Example}{Examples}
\crefname{problem}{Problem}{Problems}
\crefname{xca}{Exercise}{Exercises}
\crefname{sol}{Solution}{Solutions}
\crefname{rem}{Remark}{Remarks}
\crefname{claim}{Claim}{Claims}
\crefname{fact}{Fact}{Facts}
\crefname{criterion}{Criterion}{Criteria}
\crefname{lyxalgorithm}{Algorithm}{Algorithms}
\crefname{ax}{Axiom}{Axioms}
\crefname{condition}{Condition}{Conditions}
\crefname{note}{Note}{Notes}
\crefname{notation}{Notation}{Notations}
\crefname{summary}{Summary}{Summaries}
\crefname{acknowledgement}{Acknowledgement}{Acknowledgements}
\crefname{conclusion}{Conclusion}{Conclusions}
\crefname{assumption}{Assumption}{Assumptions}
\crefname{question}{Question}{Questions}

\makeatother

\usepackage{babel}
\providecommand{\acknowledgementname}{Acknowledgement}
\providecommand{\algorithmname}{Algorithm}
\providecommand{\assumptionname}{Assumption}
\providecommand{\axiomname}{Axiom}

\providecommand{\claimname}{Claim}
\providecommand{\conclusionname}{Conclusion}
\providecommand{\conditionname}{Condition}
\providecommand{\conjecturename}{Conjecture}
\providecommand{\corollaryname}{Corollary}
\providecommand{\criterionname}{Criterion}
\providecommand{\definitionname}{Definition}
\providecommand{\examplename}{Example}
\providecommand{\exercisename}{Exercise}
\providecommand{\factname}{Fact}
\providecommand{\lemmaname}{Lemma}
\providecommand{\notationname}{Notation}
\providecommand{\notename}{Note}
\providecommand{\problemname}{Problem}
\providecommand{\propositionname}{Proposition}
\providecommand{\questionname}{Question}
\providecommand{\remarkname}{Remark}
\providecommand{\solutionname}{Solution}
\providecommand{\summaryname}{Summary}
\providecommand{\theoremname}{Theorem}

\makeatother

\addto\captionsbritish{\renewcommand{\assumptionname}{Assumption}}
\addto\captionsbritish{\renewcommand{\corollaryname}{Corollary}}
\addto\captionsbritish{\renewcommand{\definitionname}{Definition}}
\addto\captionsbritish{\renewcommand{\lemmaname}{Lemma}}
\addto\captionsbritish{\renewcommand{\remarkname}{Remark}}
\addto\captionsbritish{\renewcommand{\theoremname}{Theorem}}
\addto\captionsenglish{\renewcommand{\assumptionname}{Assumption}}
\addto\captionsenglish{\renewcommand{\corollaryname}{Corollary}}
\addto\captionsenglish{\renewcommand{\definitionname}{Definition}}
\addto\captionsenglish{\renewcommand{\lemmaname}{Lemma}}
\addto\captionsenglish{\renewcommand{\remarkname}{Remark}}
\addto\captionsenglish{\renewcommand{\theoremname}{Theorem}}
\providecommand{\assumptionname}{Assumption}
\providecommand{\corollaryname}{Corollary}
\providecommand{\definitionname}{Definition}
\providecommand{\lemmaname}{Lemma}
\providecommand{\remarkname}{Remark}
\providecommand{\theoremname}{Theorem}

\begin{document}
\global\long\def\MT{\clubsuit}%
\global\long\def\MS{\blacklozenge}%
\global\long\def\N{\mathbb{N}}%
\global\long\def\Z{\mathbb{Z}}%
\global\long\def\Q{\mathbb{Q}}%
\global\long\def\R{\mathbb{R}}%
\global\long\def\C{\mathbb{C}}%
\global\long\def\P{\mathbb{P}}%
\global\long\def\E{\mathbb{E}}%
\global\long\def\1{\mathbbm1}%
\global\long\def\as{\ a.s.}%
\global\long\def\Pas{\:\P\text{-}a.s.}%
\global\long\def\F{F}%
\global\long\def\d{\mathrm{d}}%
\global\long\def\e{\mathrm{e}}%
\global\long\def\diff#1#2{\frac{\mathrm{d}#1}{\mathrm{d}#2}}%
\global\long\def\lebesgue{\lambda\mkern-13mu  \lambda}%
\global\long\def\argmin#1{\operatorname*{arg\,\min}_{#1}}%
\global\long\def\supp{\operatorname*{supp}}%
\global\long\def\eps{\varepsilon}%
\global\long\def\theta{\vartheta}%
\global\long\def\Span{\operatorname{span}}%
\global\long\def\leq{\leqslant}%
\global\long\def\geq{\geqslant}%
\global\long\def\le{\leqslant}%
\global\long\def\ge{\geqslant}%
\foreignlanguage{english}{}
\global\long\def\tilde#1{\widetilde{#1}}%
\foreignlanguage{english}{}
\global\long\def\hat#1{\widehat{#1}}%

\title{PAC-Bayes Bounds for High-Dimensional Multi-Index Models with Unknown
Active Dimension}
\author{Maximilian F. Steffen\thanks{Financial support of the DFG through project TR 1349/3-1 is gratefully
acknowledged. Large parts of this research were carried out while
the author was affiliated to Universität Hamburg. The author would
like to thank Mathias Trabs for helpful comments.}}
\date{Karlsruhe Institute of Technology}
\maketitle
\begin{abstract}
\noindent The multi-index model with sparse dimension reduction matrix
is a popular approach to circumvent the curse of dimensionality in
a high-dimensional regression setting. Building on the single-index
analysis by Alquier, P. \& Biau, G. (Journal of Machine Learning Research
14 (2013) 243--280), we develop a PAC-Bayesian estimation method
for a possibly misspecified multi-index model with unknown active
dimension and an orthogonal dimension reduction matrix. Our main result
is a non-asymptotic oracle inequality, which shows that the estimation
method adapts to the active dimension of the model, the sparsity of
the dimension reduction matrix and the regularity of the link function.
Under a Sobolev regularity assumption on the link function the estimator
achieves the minimax rate of convergence (up to a logarithmic factor)
and no additional price is paid for the unknown active dimension.
\end{abstract}

inequality, dimension reduction

\section{Introduction\label{sec:intro}}

A standard task in supervised learning is to estimate or learn, respectively,
the conditional expectation of a label $Y\in\R$ given a large vector
$\mathbf{X}\in\R^{p}$ of explanatory random variables based on i.i.d.
data $(\mathbf{X}_{i},Y_{i})_{i=1,\dots,n}$ which are distributed
as $(\mathbf{X},Y)$. The corresponding nonparametric regression model
reads as 
\begin{equation}
Y=\F(\mathbf{X})+\varepsilon\label{eq:regression}
\end{equation}
with an observation error $\varepsilon$ satisfying $\E[\varepsilon|\mathbf{X}]=0$
a.s. and the unknown regression function $\F\colon\R^{p}\to\R$ given
by $F=\E[Y|\mathbf{X}=\cdot]$. If the dimension $p$ is large, the
estimation problem suffers from the well-known curse of dimensionality.
This is of particular importance in numerous recent applications where
$p$ may  exceed the sample size $n$. In this paper, we provide a
fully data driven complete calibration of the high-dimensional multi-index
model with unknown active dimension.

A popular approach to reduce the effective dimension of the model
is to impose a multi-index structure on $F$ (\citealp{li1991}).
While we do not assume that the observations exactly follow a multi-index
model, our method builds upon an approximation of the regression function
of the form

\begin{equation}
\F(\mathbf{x})\approx f^{*}(\Theta^{*}\mathbf{x}),\qquad\forall\mathbf{x}\in\R^{p},\label{eq:multiindex}
\end{equation}
for some \emph{active dimension} $d^{*}\ll p$, a sparse \emph{dimension
reduction matrix} $\Theta^{*}\in\R^{d^{*}\times p}$ and a (measurable)
\emph{link function} $f^{*}\colon\R^{d^{\ast}}\to\R$. Following the
aforementioned \citet{li1991}, the estimation of the space spanned
by the rows of $\Theta^{\ast}$ has been studied extensively in the
literature, see e.g. \citet{Hristache2001}, \citet{Xia2007} and
\citet{Dalalyan2008}, but under the assumption of a known active
dimension $d^{\ast}$. While some research has been done on the estimation
of $d^{\ast}$ itself, see \citet{Xia2002} and \citet{Zhu2006},
the estimation of the overall model has relied on estimating $\Theta^{\ast}$
and $f^{\ast}$ separately to then analyse the propagation error,
see \citet{Klock2021}. The analysis of high-dimensional multi-index
models, where $p\gg n$, is rather limited.

We use a PAC-Bayesian estimation approach, see \citet{Guedj2019}
and \citet{Alquier2021} for an overview, which was originally developed
by \citet{catoni2004,catoni2007} and has been adapted to the single-index
model (i.e. $d^{\ast}=1$) without misspecification by \citet{alquier2013}.
In this paper we generalise the PAC-Bayes method for single index
models to the more flexible class of multi-index models. In particular,
we aim for a method which adapts to the unknown active dimension $d^{\ast}$,
the sparsity of $\Theta^{\ast}$ and the regularity of $f^{\ast}$
to achieve a good approximation \ref{eq:multiindex} based on the
given data. 

As a standard assumption in the theory of multi-index models, we suppose
that the dimension reduction matrix is (semi-)orthogonal, i.e. $\Theta^{\ast}(\Theta^{\ast})^{\top}=I_{d^{*}\times d^{*}}$
is the identity matrix, see \citet[Proposition 1.1]{Xia2008}. Indeed,
this allows for the interpretation of $\Theta^{\ast}\mathbf{X}$ as
a rotation of the covariates, projected onto the first $d^{\ast}$
coordinates followed by another rotation. 

\medskip{}

The PAC-Bayes approach relies on the following principle: With a prior
$\pi$ for the parameters $(d,\Theta,f)$ we consider the Gibbs-posterior
probability distribution $\hat{\rho}_{\lambda}$ whose $\pi$-density
is (up to normalisation) given by 
\begin{equation}
\diff{\hat{\rho}_{\lambda}}{\pi}(d,\Theta,f)\propto\exp\big(-\lambda R_{n}(d,\Theta,f)\big)\label{eq:posterior}
\end{equation}
with a tuning parameter $\lambda>0$ and empirical prediction risk
\[
R_{n}(d^{*},\Theta^{*},f^{*})=\frac{1}{n}\sum_{i=1}^{n}\big(Y_{i}-f^{*}(\Theta^{*}\mathbf{X}_{i})\big)^{2}.
\]
The estimator for $(d^{\ast},\Theta^{\ast},f^{\ast})$ is obtained
by simulating a random variable 
\begin{equation}
(\hat d_{\lambda},\hat{\Theta}_{\lambda},\hat f_{\lambda})\sim\hat{\rho}_{\lambda}.\label{eq:Estimator}
\end{equation}
While \ref{eq:posterior} coincides with the classical Bayesian posterior
distribution only if $Y_{i}=f(\Theta\mathbf{X}_{i})+\eps_{i}$ with
i.i.d. $\eps_{i}\sim\mathrm{N}(0,n/(2\lambda))$, the estimator $\hat F_{\hat d}:=\hat f_{\lambda}(\hat{\Theta}_{\lambda}\cdot)$
will achieve a small prediction error under quite mild model assumptions. 

We will choose a sieve prior that prefers models with a low active
dimension, sparse dimension reduction matrices and regular link functions.
Let $\pi$ be supported on $\bigcup_{d=1}^{p}\{d\}\times\mathcal{S}_{d}\times\mathcal{F}_{d}$
for some classes $\mathcal{S}_{d}$ and $\mathcal{F}_{d}$ for $\Theta$
and $f$, respectively. For $\mathcal{S}_{d}$ we will study a class
of sparse matrices while $\mathcal{F}_{d}$ will be given by finite
dimensional wavelet approximations. The prior is uniform for a given
sparsity and a wavelet projection level. The posterior weighs each
triplet of parameters $(d,\Theta,f)$ based on its empirical performance
(with respect to the empirical loss function) on the data, where the
tuning parameter $\lambda$ determines the impact of $R_{n}(d,\Theta,f)$
in comparison to the prior beliefs. 

\medskip{}

We will quantify the accuracy of the estimation procedure in terms
of the excess risk 
\begin{equation}
\mathcal{E}(d^{*},\Theta^{*},f^{*})\coloneqq R(d^{*},\Theta^{*},f^{*})-\min_{d,\Theta,f}R(d,\Theta,f)=\E[(f^{*}(\Theta^{*}\mathbf{X})-\F(\mathbf{X}))^{2}],\label{eq:excess}
\end{equation}
where
\begin{equation}
R(d^{*},\Theta^{*},f^{*})\coloneqq\E\big[(Y-f^{*}(\Theta^{*}\mathbf{X}))^{2}\big]\label{eq:lossfunction}
\end{equation}
is the prediction risk and the minimum is attained at $\Theta^{*}=I_{p\times p}$
and $f^{*}=F$. We will prove an oracle inequality verifying that
the PAC-Bayes estimator is not worse than the optimal choices for
$\Theta\in\mathcal{S}_{d}$ and $f\in\mathcal{F}_{d}$ for any $d$.
In particular, the overall quality of the method depends on the approximation
properties of the spaces $\ensuremath{\mathcal{S}_{d}}$ and $\mathcal{F}_{d}$. 

\medskip{}

The paper is organised as follows: In \ref{sec:estimation}, we explain
our estimation method. In \ref{sec:Oracle-inequalities}, the main
results are stated. The proofs have been postponed to \ref{sec:Proofs}.

\section{Construction of the prior\label{sec:estimation}}

To construct the prior, we will introduce for any dimension $d=1,\dots,p$
classes $\mathcal{S}_{d}$ and $\mathcal{F}_{d}$ together with priors
$\mu_{d}$ and $\nu_{d}$ for the dimension reduction matrix $\Theta$
and the link function $f$, respectively. Based on that we can then
define the prior $\pi$ on $\bigcup_{d=1}^{p}\{d\}\times\mathcal{S}_{d}\times\mathcal{F}_{d}$. 

We start for a fixed active dimension $d\in\{1,\dots,p\}$. While
we have mentioned above that an optimal dimension reduction matrix
$\Theta^{*}$ should be orthogonal, we will not impose this restriction
for the estimation method. Instead, we only require that the candidate
matrices have $\Vert\cdot\Vert_{2}$-normed rows, i.e. for $\Theta=(\vartheta_{1},\dots,\vartheta_{d})^{\top}\in\R^{d\times p}$
with row vectors $\vartheta_{i}=(\vartheta_{i,1},\dots,\vartheta_{i,p})\in\R^{p}$
we impose $\Vert\vartheta_{i}\Vert_{2}\coloneqq\big(\sum_{j=1}^{p}\vartheta_{i,j}^{2}\big)^{1/2}=1$.
To encode sparsity, let 
\[
\mathcal{I}_{d}\coloneqq\big\{ I\,\big|\,\emptyset\neq I:=I_{1}\times\dots\times I_{d},\,I_{1},\dots,I_{d}\subseteq\{1,\dots,p\}\big\}
\]
contain all potential sets of \emph{active coordinates,} that is $I_{i}$
describes the active coordinates in the $i$-th argument of the link
function. For $I=I_{1}\times\dots\times I_{d}\in\mathcal{I}_{d}$
the number of active coordinates is $\Vert I\Vert\coloneqq\sum_{i=1}^{d}\vert I_{i}\vert$,
where $\vert I_{i}\vert$ denotes the cardinality of $I_{i}$. Note
that $\emptyset\neq I=I_{1}\times\dots\times I_{d}$ already implies
$I_{1},\dots,I_{d}\neq\emptyset$. The parameter set $\mathcal{S}_{d}(I)$
of sparse dimension reduction matrices is given by 
\begin{align*}
\mathcal{S}_{d}(I) & \coloneqq\big\{\Theta=(\vartheta_{1},\dots,\vartheta_{d})^{\top}\in\R^{d\times p}\mid\vartheta_{i}\in\mathcal{S}(I_{i}),\,i=1,\dots,d\big\},\qquad\text{where}\\
\mathcal{S}(I_{i}) & \coloneqq\big\{\vartheta_{i}=(\vartheta_{i,1},\dots,\vartheta_{i,p})\in\R^{p}\mid\Vert\vartheta_{i}\Vert_{2}=1,\forall j\notin I_{i}:\vartheta_{i,j}=0\big\}.
\end{align*}
Finally, we define $\mathcal{S}_{d}=\bigcup_{I\in\mathcal{I}_{d}}\mathcal{S}_{d}(I)$.

Note that $\mathcal{S}_{d}(I)\supseteq\tilde S_{d}(I)$ for 
\[
\tilde{\mathcal{S}}_{d}(I)\coloneqq\{\Theta=(\vartheta_{1},\dots,\vartheta_{d})^{\top}\in\R^{d\times p}\mid\forall i=1,\dots,d:\Vert\vartheta_{i}\Vert_{2}=1,\vartheta_{i,j}\neq0\text{ iff \ensuremath{\vartheta_{i,j}}\ensuremath{\ensuremath{\in I_{i}}}},\,j=1,\dots,p\}.
\]
In $\tilde S_{d}(I)$ the index set $I$ exactly describes the sparsity
of $\Theta$. However, we consider the prior on the compact set $\mathcal{S}_{d}(I)$
to ensure the existence of solutions to minimisation problems over
$\mathcal{S}_{d}(I)$ and thus the existence of an oracle dimension
reduction matrix as a benchmark for $\hat{\Theta}$.

To construct a prior measure $\mu_{d}$ on $\mathcal{S}_{d}$, we
use the uniform distribution on the set of dimension reduction matrices
with a given active dimension $d$ and with sparsity $i=\|I\|$. These
uniform distributions are then weighted geometrically such that sparse
dimension reduction matrices are preferred by the prior. Denoting
the uniform distribution on $\mathcal{S}_{d}(I)$ by $\mu_{d,I}$,
the prior measure on $\mathcal{S}_{d}$ is thus given by the mixture
\[
\mu_{d}\coloneqq C_{\mu,d}\sum_{i=d}^{dp}10^{-i+d-1}\frac{1}{\vert\mathcal{I}_{d,i}\vert}\sum_{I\in\mathcal{I}_{d,i}}\mu_{d,I}\qquad\text{where}\qquad\mathcal{I}_{d,i}\coloneqq\{I\in\mathcal{I}_{d}\mid\Vert I\Vert=i\}
\]
and with normalisation constant $C_{\mu,d}:=9\big(1-10^{(1-p)d-1}\big)^{-1}$.
Here and in the following two analogous constructions the geometric
decay $10^{-i+d+1}$ can be replaced by $a^{-i+d+1}$ for an arbitrary
fixed $a>1$, but we choose $10$ for convenience. 

\medskip{}

To define a class $\mathcal{F}_{d}$ and a prior $\nu_{d}$ for the
link function, we will use a multivariate tensor product wavelet basis
on $\R^{d}$, see e.g. \citet{Daubechies1992,gine2016}. Let $\phi$
and $\psi$ be a continuously differentiable scaling and wavelet function
on $\R$, respectively, and write $\psi_{0}\coloneqq\phi$, $\psi_{1}\coloneqq\psi$.
We will use compactly supported regular Daubechies wavelets. For $M,N\in\N$
we define the index set
\begin{align*}
\mathcal{Z}_{M,N}^{d}\coloneqq & \{l=(0,l_{2},0)\mid l_{2}\in\Z^{d},\,\Vert l_{2}\Vert_{\infty}\leq N\}\\
 & \cup\big\{ l=(l_{1},l_{2},l_{3})\in\N_{0}\times\Z^{d}\times\{0,1\}^{d}\mid l_{1}\leq M,\Vert l_{2}\Vert_{\infty}\leq2^{l_{1}}N,l_{3}\neq0\big\},
\end{align*}
where $l_{1}$ is the approximation level, $l_{2}$ is a shift parameter
and $l_{3}$ is due to the tensor structure. The system $(\Psi_{l})_{l\in\mathcal{Z}_{\infty,\infty}^{d}}$
with
\[
\Psi_{l}(x)\coloneqq2^{l_{1}d/2}\prod_{i=1}^{d}\psi_{l_{3,i}}(2^{l_{1}}x_{i}-l_{2,i}),\qquad x\in\R^{d},l=(l_{1},l_{2},l_{3})\in\mathcal{Z}_{\infty,\infty}^{d},
\]
is an orthonormal basis of $L^{2}(\R^{d})$. In particular, each $f\in L^{2}(\R^{d})$
admits a wavelet series representation $f=\sum_{l\in\mathcal{Z}_{\infty,\infty}^{d}}\langle f,\Psi_{l}\rangle\Psi_{l}$.
Throughout, we fix a sufficiently large constant $N\in\N_{0}$ and
abbreviate $\mathcal{Z}_{M}^{d}\coloneqq\mathcal{Z}_{M,N}^{d}$. For
$\xi>0$ we define the compact wavelet coefficient ball

\begin{gather*}
\mathcal{B}_{d,M}(\xi)\coloneqq\big\{\beta\in\R^{\mathcal{Z}_{M}^{d}}\big|\Vert\beta\Vert_{\mathcal{B}}\leq\xi\big\},\qquad\text{where}\\
\Vert\beta\Vert_{\mathcal{B}}\coloneqq L^{d}\sum_{l\in\mathcal{Z}_{M}^{d}}2^{l_{1}(d/2+1)}\vert\beta_{l}\vert,\qquad L\coloneqq\Vert\psi\Vert_{\infty}\lor\Vert\phi\Vert_{\infty}\lor\Vert\psi'\Vert_{\infty}\lor\Vert\phi'\Vert_{\infty},
\end{gather*}
which determines the finite dimensional approximation space

\[
\mathcal{F}_{d,M}(\xi)\coloneqq\big\{ f=\Phi_{d,M}(\beta)\mid\beta\in\mathcal{B}_{d,M}(\xi)\big\}\qquad\text{via}\qquad\Phi_{d,M}(\beta):=\sum_{l\in\mathcal{Z}_{M}^{d}}\beta_{l}\Psi_{l},\beta\in\R^{\mathcal{Z}_{M}^{d}}.
\]
For any $f=\Phi_{d,M}(\beta)$ we write $\Vert f\|_{\mathcal{B}}\coloneqq\Vert\beta\Vert_{\mathcal{B}}$
which corresponds to the Besov norm with regularity $1+d$ and integrability
parameter $1$ on $\Span\{\Psi_{l}:l\in\mathcal{Z}_{M}^{d}\}.$ In
particular, we have for any $f\in\mathcal{F}_{d,M}(\xi)$
\begin{equation}
\Vert f\Vert_{\infty}\leq\Vert f\Vert_{\mathcal{B}}\leq\xi\qquad\text{and}\qquad\Vert(\nabla f)_{i}\Vert_{\infty}\le\|f\|_{\mathcal{B}}\leq\xi,\qquad\forall i\in\{1,\dots,d\}.\label{eq:fandgradient}
\end{equation}
For $C>0$ we set $\mathcal{F}_{d}:=\bigcup_{M=0}^{n}\mathcal{F}_{d,M}(C+1)$.

The prior $\nu_{d}$ on $\mathcal{F}_{d}$ is defined as a random
coefficient prior with uniformly distributed coefficients on $\mathcal{F}_{d,M}(C+1)$
and geometrically decreasing weights in the approximation level $M$.
To this end, let $\tilde{\nu}_{d,M}$ be the uniform distribution
on $\mathcal{B}_{d,M}(C+1)$ and let $\nu_{d,M}:=\tilde{\nu}_{d,M}(\Phi_{d,M}^{-1}(\cdot))$
denote the push-forward measure of $\tilde{\nu}_{d,M}$ under $\Phi_{d,M}$.
Then
\[
\nu_{d}\coloneqq C_{\nu,d}\sum_{M=0}^{n}10^{-M}\nu_{d,M},\qquad C_{\nu,d}:=\frac{9}{10-10^{-n}}.
\]

\medskip{}

We can now define the prior for a fixed active dimension $d$ as the
product measure $\pi_{d}\coloneqq\delta_{d}\otimes\mu_{d}\otimes\nu_{d}$
with the Dirac measure $\delta_{d}$ in $d$. Finally, we mix over
all possible active dimensions to account for the fact that $d^{\ast}$
is unknown. Encoding a preference for simple models, i.e. small active
dimensions, via weights $10^{-d}$, the final prior on $\bigcup_{d=1}^{p}\{d\}\times\mathcal{S}_{d}\times\mathcal{F}_{d}$
is given by 
\[
\pi=C_{\pi}\sum_{d=1}^{p}10^{-d}\pi_{d},\qquad C_{\pi}:=\frac{9}{1-10^{-p}}.
\]
The product measure with the Dirac measure in $d$ ensures that a
simulation of $\pi$ will yield a link function and a dimension reduction
matrix with matching active dimension.

\section{Oracle inequalities\label{sec:Oracle-inequalities}}

For an active dimension $d\in\{1,\dots,p\}$, an active index set
$I\in\mathcal{I}_{d}$ of the dimension reduction matrix and an approximation
level $M\in\{0,\dots,n\}$ of the link function, we define the \emph{oracle
choice} on $\mathcal{S}_{d}(I)\times\mathcal{F}_{d,M}(C)$ as
\begin{equation}
(\Theta_{d,I}^{\ast},f_{d,M}^{\ast})\coloneqq\argmin{(\Theta,f)\in\mathcal{S}_{d}(I)\times\mathcal{F}_{d,M}(C)}R(d,\Theta,f)\label{eq:oracle}
\end{equation}
which is not accessible to the practitioner since $R(d,\Theta,f)$
depends on the unknown distribution of $(\mathbf{X},Y)$. Note that
the minimisation in $f$ is over $\mathcal{F}_{d,M}(C)$, whereas
the prior is defined on $\mathcal{F}_{d,M}(C+1)$ which ensures that
a small neighbourhood of $f_{d,M}^{\ast}$ is contained in the support
of the prior. A solution to the minimisation problem in \ref{eq:oracle}
always exists since we have equivalently 
\[
(\Theta_{d,I}^{\ast},\beta_{d,M}^{\ast})=\argmin{(\Theta,\beta)\in\mathcal{S}_{d}(I)\times\mathcal{B}_{d,M}(C)}\E\big[\big(Y-\Phi_{d,M}(\beta)(\Theta\mathbf{X})\big)^{2}\big]
\]
with compact $\mathcal{S}_{d}(I)\times\mathcal{B}_{d,M}(C)$ and continuous
$(\Theta,\beta)\mapsto\E\big[\big(Y-\Phi_{d,M}(\beta)(\Theta\mathbf{X})\big)^{2}\big]$.
If there is more than one solution, we choose one of them. Our main
result gives a theoretical guarantee that the PAC-Bayes estimator
$(\hat d_{\lambda},\hat{\Theta}_{\lambda},\hat f_{\lambda})$ from
\ref{eq:Estimator} is at least as good as the best oracle $(\Theta_{d,I}^{\ast},f_{d,M}^{\ast})_{d,I,M}$
in terms of the excess risk. To this end, we need some mild assumptions
on the regression model \ref{eq:regression}.
\begin{assumption}
\label{assu:bounded}~
\begin{enumerate}
\item For $K,C\ge1$ we have $\Vert\mathbf{X}\Vert_{\infty}\leq K$ a.s.
and $\Vert\F\Vert_{\infty}\leq C$.
\item $\eps$ is conditionally on $\mathbf{X}$ sub-Gaussian, i.e. there
are constants $\sigma,\Gamma>0$ such that 
\[
\E[\vert\varepsilon\vert^{k}|\mathbf{X}]\leq\frac{k!}{2}\sigma^{2}\Gamma^{k-2}\as,\qquad\forall k\geq2.
\]
\end{enumerate}
\end{assumption}

We obtain the following non-asymptotic oracle inequality. It generalises
\citet[Theorem 2]{alquier2013} not only with respect to the multi-index
approach with unknown active dimension, but also with respect to some
technical but practically relevant aspects like the $\ell^{2}$-normalisation
of $\Theta$ and the wavelet basis.
\begin{thm}
\label{thm:oracle}Under \ref{assu:bounded} with $Q=8(2C+1)(\Gamma\lor(2C+1))$
set 

\[
\lambda=\frac{n}{Q+2((2C+1)^{2}+4\sigma^{2})}.
\]
If $n\geq(10C^{-1})\lor(30\sqrt{2}\e p^{-2})$, then we have for any
$\delta\in(0,1)$

\[
\mathcal{E}(\hat d_{\lambda},\hat{\Theta}_{\lambda},\hat f_{\lambda})\leq\inf_{d,I,M}\bigg(3\mathcal{E}(d,\Theta_{d,I}^{\ast},f_{d,M}^{\ast})+\frac{\Xi}{n}\Big(\Vert I\Vert\log(pn)+4^{d}N^{d}2^{dM}\log(Cn)+\log\big(\frac{2}{\delta}\text{\ensuremath{\big)}}\Big)\bigg)
\]
with $(\P\otimes\hat{\rho}_{\lambda})$-probability of at least $1-\delta$,
where $\Xi$ is a constant only depending on $C,K,\Gamma$ and $\sigma$
and the infimum is taken over all triplets $(d,I,M)$ with $d\in\{1,\dots,p\}$,
$I\in\mathcal{I}_{d}$ and $M\in\{0,\dots,n\}$.
\end{thm}

Some remarks are in order: \textbf{1) }The $1-\delta$ probability
with respect to the coupling $\P\otimes\hat{\rho}_{\lambda}$ takes
into account the randomness of the data and of the estimate.  \textbf{2)
}The right-hand side of the oracle inequality can be interpreted similiar
to the classical bias-variance decomposition in non-parametric statistics.
The first term $\mathcal{E}(d,\Theta_{d,I}^{\ast},f_{d,M}^{\ast})=\E[(f_{d,M}^{*}(\Theta_{d,I}^{*}\mathbf{X})-\F(\mathbf{X}))^{2}]$
quantifies the approximation error while second term is an upper bound
for the stochastic error. In particular, we recover $\frac{\|I\|}{n}\log(pn)$
(or $\frac{\|I\|}{n}\log(p)$ if $p\ge n$) as the typical error term
for estimating sparse matrices with sparsity $\|I\|$ while $4^{d}N^{d}2^{dM}\frac{\log(n)}{n}$
is due to the estimation of $4^{d}N^{d}2^{dM}$ many wavelet coefficients
each with (squared) accuracy $\frac{\log(n)}{n}$ paying a logarithmic
price for adaptiveness.\textbf{ 3)} The infimum over all $(d,I,M)$
in the upper bound shows that the estimator adapts to the active dimension,
the sparsity of the dimension reduction matrix and the regularity
of the link function.\textbf{ 4) }The factor $3$ in the upper bound
can be improved to $(1+\tau)$ for any $\tau>0$ at the cost of a
larger constant $\Xi$.  \textbf{5) }One can show the same result
in a multi-index model with a known active dimension $d^{\ast}$ by
using $\pi_{d^{\ast}}$ as a prior instead of $\pi$. The only difference
(up to a constant) in the result is that the infimum in the upper
bound is only taken over all pairs $(I,M)\in\mathcal{I}_{d^{\ast}}\times\{0,\dots,n\}$.
Consequently, no additional price is paid for not knowing the true
active dimension of the model.

\medskip{}

In the well-specified setting and under assumptions on the distribution
of $\Theta^{\ast}\mathbf{X}$ as well as a Besov-type regularity assumption
on the link function, we can make the upper bound from \ref{thm:oracle}
more explicit.
\begin{assumption}
\label{assu:nonmisspecified}~
\begin{enumerate}
\item There exist $d^{\ast}\in\{1,\dots,p\}$, $\Theta^{\ast}\in\mathcal{S}^{d^{\ast}}$
and $f^{\ast}\colon\R^{d^{\ast}}\to\R$ such that $\F=f^{\ast}(\Theta^{\ast}\cdot)$.
\item The random variable $\Theta^{\ast}\mathbf{X}$ is bounded, i.e. $\Vert\Theta^{\ast}\mathbf{X}\Vert_{\infty}\leq B_{1}$
for some $B_{1}>0$, and $\Theta^{\ast}\mathbf{X}$ has a $\lebesgue^{d^{\ast}}$-density
on $\R^{d^{\ast}}$ bounded by a constant $B_{2}>0$. 
\end{enumerate}
\end{assumption}

For the true dimension reduction matrix $\Theta^{*}$ we write $\Vert\Theta^{\ast}\Vert_{0}\coloneqq\Vert I^{\ast}\Vert$
for the minimal $I^{\ast}\in\mathcal{I}_{d^{\ast}}$ with respect
to $\Vert\cdot\Vert$ such that $\Theta^{\ast}\in\mathcal{S}_{d^{\ast}}(I^{\ast})$.
The regularity of $f^{*}$ will be measured in terms of its Besov
norm. We recover Sobolev balls for $q=2$, cf. \citet[(4.164)]{gine2016}.
\begin{defn}
The Besov ellipsoid in $\R^{d^{*}}$ with regularity $\alpha>0$ and
integrability parameter $q\in[0,\infty)$ is given by 
\begin{equation}
B_{q,d^{\ast}}^{\alpha}(\xi)\coloneqq\Big\{ f\in L^{2}(\R^{d^{\ast}})\,\Big|\,\sum_{l\in\mathcal{Z}_{\infty,\infty}^{d^{*}}}2^{ql_{1}\alpha}\vert\langle f,\Psi_{l}\rangle\vert^{q}\leq\xi^{q}\Big\}\label{eq:besov}
\end{equation}
for a radius $\xi>0$. 
\end{defn}

\begin{cor}
\label{cor:oracleineq}Let the assumptions of \ref{thm:oracle} be
fulfilled in addition to \ref{assu:nonmisspecified}. Suppose that
$f^{\ast}\in B_{d^{\ast}}^{\alpha}(\xi)$ with $\xi=C(L^{d^{\ast}}2N^{d^{\ast}/2}8^{d^{\ast}/2})^{-1}$
for some $\alpha>2+d^{\ast}$ and $n\geq2$. Then, with $(\P\otimes\hat{\rho}_{\lambda})$-probability
of at least $1-\delta$, we have

\[
\mathcal{E}(\hat d_{\lambda},\hat{\Theta}_{\lambda},\hat f_{\lambda})\leq\Lambda\bigg(\Big(\frac{\log(Cn)}{n}\Big)^{\frac{2\alpha}{2\alpha+d^{\ast}}}+\frac{\Vert\Theta^{\ast}\Vert_{0}\log(pn)}{n}+\frac{\log(\frac{2}{\delta})}{n}\bigg),
\]
where $\Lambda$ only depending on $C,K,\Gamma,\sigma,N,B_{1},B_{2}$
and $d^{\ast}$. 
\end{cor}

\begin{rem*}
If $\Theta^{\ast}$is sparse (i.e. $\Vert\Theta^{\ast}\Vert_{0}$
is small), then the dominating term in the upper bound of the excess
risk of the PAC-Bayesian estimator is of order
\[
\Big(\frac{\log(n)}{n}\Big)^{\frac{2\alpha}{2\alpha+d^{\ast}}},
\]
which is the usual minimax rate (up to a logarithmic factor) for such
estimation problems.  Note that if $d^{\ast}\ll p$, then we have
successfully circumvented the curse of dimensionality, since the dimension
which appears in the rate is now only $d^{\ast}$. As an alternative
to the wavelet construction, one can use the multivariate trigonometric
system on $[-1,1]^{d^{\ast}}$, assume $\mathbf{X}\in[-1,1]^{p}$
and $\ell^{1}$-standardised rows of $\Theta^{\ast}$ (which ensures
$\Theta^{\ast}\mathbf{X}\in[-1,1]^{d^{\ast}}$) leading to a more
direct generalisation of \citealt{alquier2013}. However, the orthogonality
assumption on $\Theta^{\ast}$ seems more natural and in line with
the literature.
\end{rem*}

\section{Proofs\label{sec:Proofs}}

We begin with a few preliminary results. The first is a classical
result on the Kullback-Leibler divergence. A proof can be found in
\citet[p. 159]{catoni2004}. Afterwards, we present two lemmata of
the \textquotedblleft PAC\textquotedblright -type, which are extensions
of \citet[Lemmata 7 \& 8]{alquier2013}. The proofs can be carried
out analogously.

Let $\mu,\nu$ be probability measures on a measurable space $(E,\mathscr{A})$.
The \emph{Kullback-Leibler divergence} of $\mu$ with respect to $\nu$
is defined via 
\[
\mathcal{K}(\mu,\nu)\coloneqq\begin{cases}
\int\log\big(\diff{\mu}{\nu}\big)\d\mu, & \text{if }\mu\ll\nu\\
\infty, & \text{otherwise}
\end{cases}.
\]

\begin{lem}
\label{lem:classicallemma}Let $\mu$ be a probability measure on
a measurable space $(E,\mathscr{A})$ and let $h\colon E\to\R$ be
a measurable function such that$\int\exp\circ h\d\mu<\infty$. With
the convention $\infty-\infty=-\infty$, it then holds that 
\begin{equation}
\log\Big(\int\exp\circ h\d\mu\Big)=\sup_{\nu}\Big(\int h\d\nu-\mathcal{K}(\nu,\mu)\Big)=-\inf_{\nu}\Big(\mathcal{K}(\nu,\mu)-\int h\d\nu\Big),\label{eq:gibbsequality}
\end{equation}
where the supremum and the infimum are taken over all probability
measures $\nu$ on $(E,\mathscr{A})$, but no generality is lost by
considering only those probability measures $\nu$ on $(E,\mathscr{A})$
such that $\nu\ll\mu$. If additionally, $h$ is bounded from above
on the support of $\mu$, then the supremum and the infimum in \ref{eq:gibbsequality}
are both attained by taking $\nu=g$ with the Gibbs-posterior probability
distribution $g$, i.e. $\diff g{\mu}\propto\exp\circ h$.
\end{lem}

\begin{lem}
\label{lem:pacexcess}Under \ref{assu:bounded}, let $Q=8(2C+1)(\Gamma\lor(2C+1))$,
$\mathcal{E}_{n}(d,\Theta,f)\coloneqq R_{n}(d,\Theta,f)-R_{n}(d^{\ast},\Theta^{\ast},f^{\ast})$
and take 
\[
\lambda\in\Big(0,\frac{n}{Q+((2C+1)^{2}+4\sigma^{2})}\Big).
\]
Then, for all $\delta\in(0,1)$ and any probability measure $\rho\ll\pi$,
we have that 
\[
\mathcal{E}(\hat d_{\lambda},\hat{\Theta}_{\lambda},\hat f_{\lambda})\leq\frac{1}{1-\frac{\lambda((2C+1)^{2}+4\sigma^{2})}{n-Q\lambda}}\Big(\int\mathcal{E}_{n}\d\rho+\frac{\mathcal{K}(\rho,\pi)+\log(\frac{1}{\delta})}{\lambda}\Big)
\]
with $(\P\otimes\hat{\rho}_{\lambda})$-probability of at least $1-\delta$,
where $(\hat d_{\lambda},\hat{\Theta}_{\lambda},\hat f_{\lambda})\sim\hat{\rho}_{\lambda}$.
\end{lem}

\begin{lem}
\label{lem:pacexcessempirical}In the situation of \ref{lem:pacexcess},
we have 
\[
\int\mathcal{E}_{n}\d\rho\leq\Big(1+\frac{\lambda((2C+1)^{2}+4\sigma^{2})}{n-Q\lambda}\Big)\int\mathcal{E}\d\rho+\frac{\mathcal{K}(\rho,\pi)+\log(\frac{1}{\delta})}{\lambda}
\]
 with $(\P\otimes\hat{\rho}_{\lambda})$-probability of at least $1-\delta$.
\end{lem}

\subsection[Proof of Theorem 2]{Proof of \ref{thm:oracle}}

We extend the proof strategy by \citet{alquier2013} to the multi-index
setting with unknown active dimension.

We begin by fixing some triplet $(d,I,M)$ with $d\in\{1,\dots,p\}$,
$I=I_{1}\times\cdots\times I_{d}\in\mathcal{I}_{d}$, $M\in\{0,\dots,n\}$
as well as $\eta,\gamma\in(0,1]$ and introduce the probability measure
\begin{equation}
\rho\coloneqq\rho_{d,I,M,\eta,\gamma}\coloneqq\kappa_{d}\otimes\rho_{d,I,\eta}^{1}\otimes\rho_{d,M,\gamma}^{2},\label{eq:overallrho}
\end{equation}
where $\rho_{d,I,\eta}^{1}$ and $\rho_{d,M,\gamma}^{2}$ are the
uniform distribution with respect to $\mu_{d,I}$ and $\nu_{d,M}$
on a ball of radius $\eta$ and $\gamma$ around the oracle $\Theta_{d.I}^{\ast}$
and $f_{d,M}^{\ast}$, respectively. Specifically, we set
\begin{align}
\diff{\rho_{d,I,\eta}^{1}}{\mu_{d,I}}(\Theta) & \coloneqq\prod_{i=1}^{d}\diff{\rho_{d,I,\eta}^{1,i}}{\mu_{I_{i}}}(\vartheta_{i}),\,\forall\Theta=(\vartheta_{1},\dots,\vartheta_{d})^{\top},\quad\text{where}\label{eq:rhotheta}\\
\diff{\rho_{d,I,\eta}^{1,i}}{\mu_{I_{i}}}(\vartheta_{i}) & \propto\1_{\{\Vert\vartheta_{i}-\vartheta_{d,I}^{\ast}\Vert_{2}\leq\eta\}}\quad\text{and}\nonumber \\
\diff{\rho_{d,M,\gamma}^{2}}{\nu_{d,M}}(f) & \propto\1_{\{\Vert f-f_{d,M}^{\ast}\Vert_{\psi}\leq\gamma\}},\label{eq:rhof}
\end{align}
where $\mu_{I_{i}}$denotes the uniform distribution on $\mathcal{S}(I_{i})$.
Applying \ref{lem:pacexcess,lem:pacexcessempirical}, we have that
\begin{align*}
\mathcal{E}(\hat d_{\lambda},\hat{\Theta}_{\lambda},\hat f_{\lambda}) & \leq\frac{1}{1-\frac{\lambda((2C+1)^{2}+4\sigma^{2})}{n-Q\lambda}}\Big(\int\mathcal{E}_{n}\d\rho+\frac{\mathcal{K}(\rho,\pi)+\log(\frac{2}{\delta})}{\lambda}\Big),\qquad\text{and }\\
\int\mathcal{E}_{n}\d\rho & \leq\Big(1+\frac{\lambda((2C+1)^{2}+4\sigma^{2})}{n-Q\lambda}\Big)\int\mathcal{E}\d\rho+\frac{\mathcal{K}(\rho,\pi)+\log(\frac{2}{\delta})}{\lambda}
\end{align*}
with a probability of at least $1-\frac{\delta}{2}$, respectively.
Therefore,
\begin{equation}
\mathcal{E}(\hat d_{\lambda},\hat{\Theta}_{\lambda},\hat f_{\lambda})\leq\frac{1}{1-\frac{\lambda((2C+1)^{2}+4\sigma^{2})}{n-Q\lambda}}\Big(\Big(1+\frac{\lambda((2C+1)^{2}+4\sigma^{2})}{n-Q\lambda}\Big)\int\mathcal{E}\d\rho+2\frac{\mathcal{K}(\rho,\pi)+\log(\frac{2}{\delta})}{\lambda}\Big)\label{eq:pacineq}
\end{equation}
holds with a probability of at least $1-\delta$. To complete the
proof, we need to bound the terms on the right hand side of \ref{eq:pacineq}.

First, we deal with the Kullback-Leibler divergence term using the
following two lemmata:
\begin{lem}
\label{lem:klequal} For $\rho=\rho_{d,I,M,\eta,\gamma}$ from \ref{eq:overallrho}
and with $\pi_{d,I,M}=\kappa_{d}\otimes\mu_{d,I}\otimes\nu_{d,M}$,
we have

\[
\mathcal{K}(\rho,\pi)\leq\Vert I\Vert\log(\e p)+(\Vert I\Vert+M+1)\log(10)+\mathcal{\mathcal{K}}(\rho,\pi_{d,I,M})\eqqcolon T_{1}+\mathcal{\mathcal{K}}(\rho,\pi_{d,I,M}).
\]
\end{lem}

\begin{lem}
\label{lem:klequal2} For $\rho=\rho_{d,I,M,\eta,\gamma}$ from \ref{eq:overallrho}
and with $\pi_{d,I,M}=\kappa_{d}\otimes\mu_{d,I}\otimes\nu_{d,M}$,
we have
\[
\mathcal{K}(\rho,\pi_{d,I,M})\leq\Vert I\Vert\log\Big(\frac{3\sqrt{2}}{\eta}\Big)+4^{d}N^{d}2^{dM+1}\log\Big(\frac{C+1}{\gamma}\Big)\eqqcolon T_{2}.
\]
\end{lem}

Thus,

\[
\mathcal{E}(\hat d_{\lambda},\hat{\Theta}_{\lambda},\hat f_{\lambda})\leq\frac{1}{1-\frac{\lambda((2C+1)^{2}+4\sigma^{2})}{n-Q\lambda}}\Big(\Big(1+\frac{\lambda((2C+1)^{2}+4\sigma^{2})}{n-Q\lambda}\Big)\int\mathcal{E}\d\rho+2\frac{T_{1}+T_{2}+\log(\frac{2}{\delta})}{\lambda}\Big)
\]
with a probability of at least $1-\delta$.

Second, we control the integral term in \ref{eq:pacineq} by splitting
it into

\begin{align}
\int\mathcal{E}(c,\Theta,f)\d\rho(c,\Theta,f) & =\mathcal{E}(d,\Theta_{d,I}^{\ast},f_{d,M}^{\ast})\nonumber \\
 & \qquad+\int\E[(f_{c,M}^{\ast}(\Theta_{c,I}^{\ast}\mathbf{X})-f(\Theta_{c,M}^{\ast}\mathbf{X}))^{2}]\d\rho(c,\Theta,f)\nonumber \\
 & \qquad+\int\E[(f(\Theta_{c,I}^{\ast}\mathbf{X})-f(\Theta\mathbf{X}))^{2}]\d\rho(c,\Theta,f)\nonumber \\
 & \qquad+\int\E[2(Y-f_{c,M}^{\ast}(\Theta_{c,I}^{\ast}\mathbf{X}))(f_{c,M}^{\ast}(\Theta_{c,I}^{\ast}\mathbf{X})-f(\Theta_{c,I}^{\ast}\mathbf{X}))]\d\rho(c,\Theta,f)\nonumber \\
 & \qquad+\int\E[2(Y-f_{c,M}^{\ast}(\Theta_{c,I}^{\ast}\mathbf{X}))(f(\Theta_{c,I}^{\ast}\mathbf{X})-f(\Theta\mathbf{X}))]\d\rho(c,\Theta,f)\nonumber \\
 & \qquad+\int\E[2(f_{c,M}^{\ast}(\Theta_{c,I}^{\ast}\mathbf{X})-f(\Theta_{c,I}^{\ast}\mathbf{X}))(f(\Theta_{c,I}^{\ast}\mathbf{X})-f(\Theta\mathbf{X}))]\d\rho(c,\Theta,f)\nonumber \\
 & \eqqcolon\mathcal{E}(d,\Theta_{d,I}^{\ast},f_{d,M}^{\ast})+U_{1}+U_{2}+U_{3}+U_{4}+U_{5}\label{eq:splitthem}
\end{align}
and treating the terms $U_{1},\dots,U_{5}$ sequentially. Note, that
integrating with respect to $\rho=\kappa_{d}\otimes\rho_{d,I,\eta}^{1}\otimes\rho_{d,M,\gamma}^{2}$
amounts to fixing the first argument at $d$ and then integrating
with respect to $\rho_{d,I,\eta}^{1}\otimes\rho_{d,I,\gamma}^{2}$. 

Similarly to \ref{eq:fandgradient}, $f=\Phi_{d,M}(\beta)\in\mathcal{F}_{d,M}(C+1)$
with $\Vert\beta-\beta_{d.M}^{\ast}\Vert_{\mathcal{B}}\le\gamma$
implies
\[
\Vert f-f_{d,M}^{\ast}\Vert_{\infty}\leq\Vert f-f_{d,M}^{\ast}\Vert_{\psi}=\Vert\beta-\beta_{d.M}^{\ast}\Vert_{\mathcal{B}}\leq\gamma.
\]
Therefore,
\begin{equation}
U_{1}=\int\E[(f_{d,M}^{\ast}(\Theta_{d,I}^{\ast}\mathbf{X})-f(\Theta_{d,I}^{\ast}\mathbf{X}))^{2}]\d\rho_{d,M,\gamma}^{2}(f)\leq\int\sup_{x\in\R^{d}}(f_{d,M}^{\ast}(x)-f(x))^{2}\d\rho_{d,M,\gamma}^{2}(f)\leq\gamma^{2}.\label{eq:u1}
\end{equation}
Any $f\in\mathcal{F}_{d,M}(C+1)$ is differentiable as a linear combination
of only finitely many basis elements. Therefore, applying the fundamental
theorem of calculus to the mapping
\[
h\colon[-1,1]\to\R,\,s\mapsto f((\Theta+s(\Theta_{d,I}^{\ast}-\Theta))\mathbf{X}(\omega))
\]
with a fixed $\omega\in\Omega$ (which we will omit from here on)
yields
\[
f(\Theta_{d,I}^{\ast}\mathbf{X})-f(\Theta\mathbf{X})=\int_{0}^{1}g'(s)\d s=\Big\langle(\Theta_{d,I}^{\ast}-\Theta)\mathbf{X},\int_{0}^{1}\nabla f((\Theta+s(\Theta_{d,I}^{\ast}-\Theta))\mathbf{X})\d s\Big\rangle
\]
and combined with \ref{eq:fandgradient} we obtain

\begin{align}
\vert f(\Theta_{d,M}^{\ast}\mathbf{X})-f(\Theta\mathbf{X})\vert & =\Big\vert\Big\langle(\Theta_{d,I}^{\ast}-\Theta)\mathbf{X},\int_{0}^{1}\nabla f((\Theta+s(\Theta_{d,I}^{\ast}-\Theta))\mathbf{X})\d s\Big\rangle\Big\vert\nonumber \\
 & \leq\Vert(\Theta_{d,I}^{\ast}-\Theta)\mathbf{X}\Vert_{2}\Big\Vert\int_{0}^{1}\nabla f((\Theta+s(\Theta_{d,I}^{\ast}-\Theta))\mathbf{X})\d s\Big\Vert_{2}\nonumber \\
 & \leq\Vert\mathbf{X}\Vert_{2}\Big(\sum_{i=1}^{d}\Vert\vartheta_{d,I,i}^{\ast}-\vartheta_{i}\Vert_{2}^{2}\Big)^{\frac{1}{2}}\sqrt{d}C\nonumber \\
 & \leq pK\Big(\sum_{i=1}^{d}\Vert\vartheta_{d,I,i}^{\ast}-\vartheta_{i}\Vert_{2}^{2}\Big)^{\frac{1}{2}}\sqrt{d}C\Pas\label{eq:gradienttrick}
\end{align}
Using the above, we deduce
\begin{align}
U_{2} & =\int\E[(f(\Theta_{d,I}^{\ast}\mathbf{X})-f(\Theta\mathbf{X}))^{2}]\d\rho_{d,I,\eta}^{1}\otimes\rho_{d,M,\gamma}^{2}(\Theta,f)\nonumber \\
 & \leq d(pKC)^{2}\int\dots\int\sum_{i=1}^{d}\Vert\vartheta_{d,I,i}^{\ast}-\vartheta_{i}\Vert_{2}^{2}\d\rho_{d,I,\eta}^{1,1}(\vartheta_{1})\dots\d\rho_{d,I,\eta}^{1,d}(\vartheta_{d})\nonumber \\
 & \leq(dpKC\eta)^{2}.\label{eq:u2}
\end{align}
By construction, $\rho_{d,M,\gamma}^{2}$ is centered around $f_{d,M}^{\ast}$
and thus
\begin{equation}
\int f(x)\d\rho_{d,M,\gamma}^{2}(f)=f_{d,M}^{\ast}(x),\,\forall x\in\R^{p}.\label{eq:u3iszero}
\end{equation}
In particular, we have
\begin{equation}
U_{3}=0.\label{eq:u3}
\end{equation}
Using Fubini's theorem together with \ref{eq:u3iszero}, we have
\begin{align}
\vert U_{4}\vert & =2\big\vert\E\big[(Y-f_{d,M}^{\ast}(\Theta_{d,I}^{\ast}\mathbf{X}))\int\int f(\Theta_{d,I}^{\ast}\mathbf{X})-f(\Theta\mathbf{X})\d\rho_{d,M,\gamma}^{2}(f)\d\rho_{d,I,\eta}^{1}(\Theta)\big]\big\vert\nonumber \\
 & =2\big\vert\E\big[(Y-f_{d,M}^{\ast}(\Theta_{d,I}^{\ast}\mathbf{X}))\int f(\Theta_{d,I}^{\ast}\mathbf{X})-f_{d,M}^{\ast}(\Theta\mathbf{X})\d\rho_{d,I,\eta}^{1}(\Theta)\big]\big\vert\nonumber \\
 & \leq2\sqrt{R(d,\Theta_{d,I}^{\ast},f_{d,M}^{\ast})}\sqrt{\E\big[\big(\int f_{d,M}^{\ast}(\Theta_{d,I}^{\ast}\mathbf{X})-f_{d,M}^{\ast}(\Theta\mathbf{X})\d\rho_{d,I,\eta}^{1}(\Theta)\big)^{2}\big]},\label{eq:u4help1}
\end{align}
where the Cauchy-Schwarz inequality has been used in the final step.
Repeating the argument from treating $U_{2}$, but now with $f=f_{d,M}^{\ast}$,
we obtain 
\begin{equation}
\E\big[\big(\int f_{d,M}^{\ast}(\Theta_{d,I}^{\ast}\mathbf{X})-f_{d,M}^{\ast}(\Theta\mathbf{X})\d\rho_{d,I,\eta}^{1}(\Theta)\big)^{2}\big]\leq(dpKC\eta)^{2}.\label{eq:u4help2}
\end{equation}
Clearly, $f=0\in\mathcal{F}_{d,M}(C)$ and thus we have by definition
of $(\Theta_{d,I}^{\ast},f_{d,M}^{\ast})$ that
\begin{equation}
R(d,\Theta_{d,I}^{\ast},f_{d,M}^{\ast})\leq R(d,\Theta_{d,I}^{\ast},f)=\E[Y^{2}]=\E[\F(\mathbf{X})^{2}]+2\E[\F(\mathbf{X})\E[\varepsilon|\mathbf{X}]]+\E[\varepsilon^{2}]\leq C^{2}+\sigma^{2}.\label{eq:u4help3}
\end{equation}
Plugging \ref{eq:u4help2,eq:u4help3} into \ref{eq:u4help1}, we have
\begin{equation}
\vert U_{4}\vert\leq2dpKC\eta\sqrt{C^{2}+\sigma^{2}}.\label{eq:u4}
\end{equation}
Finally, applying \ref{eq:gradienttrick} again yields

\begin{align}
\vert U_{5}\vert & \leq2\int\E[\vert f_{d,M}^{\ast}(\Theta_{d,I}^{\ast}\mathbf{X})-f(\Theta_{d,I}^{\ast}\mathbf{X})\vert\vert f(\Theta_{d,I}^{\ast}\mathbf{X})-f(\Theta\mathbf{X})\vert]\d\rho_{d,I,\eta}^{1}\otimes\rho_{d,M,\gamma}^{2}(\Theta,f)\nonumber \\
 & \leq2\sqrt{d}pKC\E\Big[\int\vert f_{d,M}^{\ast}(\Theta_{d,I}^{\ast}\mathbf{X})-f(\Theta_{d,I}^{\ast}\mathbf{X})\vert\Big(\sum_{i=1}^{d}\Vert\vartheta_{d,I,i}^{\ast}-\vartheta_{i}\Vert_{2}^{2}\Big)^{\frac{1}{2}}\d\rho_{d,I,\eta}^{1}\otimes\rho_{d,M,\gamma}^{2}((\vartheta_{1},\dots,\vartheta_{d})^{\top},f)\Big]\nonumber \\
 & =2\sqrt{d}pKC\int\vert f_{d,M}^{\ast}(\Theta_{d,I}^{\ast}\mathbf{X}(\omega))-f(\Theta_{d,I}^{\ast}\mathbf{X}(\omega))\vert\nonumber \\
 & \qquad\qquad\cdot\Big(\sum_{i=1}^{d}\Vert\vartheta_{d,I,i}^{\ast}-\vartheta_{i}\Vert_{2}^{2}\Big)^{1/2}\d\P\otimes\rho_{d,I,\eta}^{1}\otimes\rho_{d,M,\gamma}^{2}(\omega,(\vartheta_{1},\dots,\vartheta_{d})^{\top},f)\nonumber \\
 & \leq2\sqrt{d}pKC\sqrt{\int\big(f_{d,M}^{\ast}(\Theta_{d,I}^{\ast}\mathbf{X}(\omega))-f(\Theta_{d,I}^{\ast}\mathbf{X}(\omega))\big)^{2}\d\P\otimes\rho_{d,I,\eta}^{1}\otimes\rho_{d,M,\gamma}^{2}(\omega,(\vartheta_{1},\dots,\vartheta_{d})^{\top},f)}\nonumber \\
 & \qquad\qquad\cdot\sqrt{\int\sum_{i=1}^{d}\Vert\vartheta_{d,I,i}^{\ast}-\vartheta_{i}\Vert_{2}^{2}\d\P\otimes\rho_{d,I,\eta}^{1}\otimes\rho_{d,M,\gamma}^{2}(\omega,(\vartheta_{1},\dots,\vartheta_{d})^{\top},f)}\label{eq:u5help}\\
 & \leq2\sqrt{d}pKC\sqrt{\int\E[(f(\Theta_{d,I}^{\ast}\mathbf{X})-f(\Theta_{d,I}^{\ast}\mathbf{X}))^{2}]\d\rho_{d,M,\gamma}^{2}(f)}\sqrt{\int\sum_{i=1}^{d}\Vert\vartheta_{d,I,i}^{\ast}-\vartheta_{i}\Vert_{2}^{2}\rho_{d,I,\eta}^{1}((\vartheta_{1},\dots,\vartheta_{d})^{\top})}\nonumber \\
 & \leq2dpKC\eta\gamma,\label{eq:u5}
\end{align}
where \ref{eq:u5help} follows from the Cauchy-Schwarz inequality
for integration with respect to the product measure $\P\otimes\rho_{d,I,\eta}^{1}\otimes\rho_{d,M,\gamma}^{2}$.

Choosing $\eta=(dpn)^{-1},\gamma=n^{-1}$ when summarising \ref{eq:u1,eq:u2,eq:u3,eq:u4,eq:u5,eq:splitthem}
we have
\begin{align*}
\int\mathcal{E}(c,\Theta,f)\d\rho(c,\Theta,f) & \leq\mathcal{E}(d,\Theta_{d,I}^{\ast},f_{d,M}^{\ast})+\gamma^{2}+(dpKC\eta)^{2}+2dpKC\eta\sqrt{C^{2}+\sigma^{2}}+2dpKC\eta\gamma\\
 & \leq\mathcal{E}(d,\Theta_{d,I}^{\ast},f_{d,M}^{\ast})+\frac{\Xi_{1}}{n},
\end{align*}
where $\Xi_{1}$ is a constant only depending on $C,K$ and $\sigma$.
With these choices for $\eta$ and $\gamma$, we can use the assumption
that $n\geq(10C^{-1})\lor(30\sqrt{2}\e p^{-2})$ together with $C\geq1$
(by \ref{assu:bounded}) to bound
\begin{align*}
T_{1}+T_{2} & =\Vert I\Vert\log(\e p)+(\Vert I\Vert+M+1)\log(10)+\Vert I\Vert\log(3\sqrt{2}dpn)+4^{d}N^{d}2^{dM+1}\log(n(C+1))\\
 & \leq\Vert I\Vert\log(\e p)+(\Vert I\Vert+M+1)\log(10)+\Vert I\Vert\log(3\sqrt{2}dpn)+4^{d}N^{d}2^{dM+1}\log(n(C+1))\\
 & \leq4\Vert I\Vert\log(pn)+16\cdot4^{d}N^{d}2^{dM}\log(Cn).
\end{align*}
Choosing 
\[
\lambda=\frac{n}{Q+2((2C+1)^{2}+4\sigma^{2})},
\]
we have 
\[
\frac{1}{1-\frac{\lambda((2C+1)^{2}+4\sigma^{2})}{n-Q\lambda}}=2\qquad\text{as well as}\qquad1+\frac{\lambda((2C+1)^{2}+4\sigma^{2})}{n-Q\lambda}=\frac{3}{2}
\]
and there exists a constant $\Xi_{2}$ depending only on $C,K,\Gamma$
and $\sigma$ such that $\frac{2}{\lambda}\leq\frac{\Xi_{2}}{n}$.
Summarizing the above, we arrive at 

\begin{equation}
\mathcal{E}(\hat d_{\lambda},\hat{\Theta}_{\lambda},\hat f_{\lambda})\leq3\mathcal{E}(d,\Theta_{d,I}^{\ast},f_{d,M}^{\ast})+\frac{\Xi}{n}\Big(\Vert I\Vert\log(pn)+4^{d}N^{d}2^{dM}\log(Cn)+\log\big(\frac{2}{\delta}\big)\Big)\label{eq:oracleineqnoinf}
\end{equation}
with a probability of at least $1-\delta$, where $\Xi$ is a constant
only depending on $C,K,\Gamma$ and $\sigma$. Note that the upper
bound in \ref{eq:oracleineqnoinf} is deterministic. Choosing a triplet
$(d,I,M)$ such that this upper bound is minimised (which is always
possible, since there are only finitely many choices for $(d,I,M)$),
we have shown that
\[
\mathcal{E}(\hat d_{\lambda},\hat{\Theta}_{\lambda},\hat f_{\lambda})\leq\inf_{d,I,M}\bigg(3\mathcal{E}(d,\Theta_{d,I}^{\ast},f_{d,M}^{\ast})+\frac{\Xi}{n}\Big(\Vert I\Vert\log(pn)+4^{d}N^{d}2^{dM}\log(Cn)+\log\big(\frac{2}{\delta}\big)\Big)\bigg),
\]
with a probability of at least $1-\delta$. This completes the proof
of \ref{thm:oracle}.\hfill\qed

\subsection[Proof of Corollary 5]{Proof of \ref{cor:oracleineq}}

Plugging in $d^{\ast}$ and $I^{\ast}$ in the infimum in \ref{thm:oracle},
we obtain that
\begin{align}
\mathcal{E}(\hat d_{\lambda},\hat{\Theta}_{\lambda},\hat f_{\lambda}) & \leq\inf_{0\leq M\leq n}\bigg(3\mathcal{E}(d^{\ast},\Theta_{d^{\ast},I}^{\ast},f_{d^{\ast},M}^{\ast})+\frac{\Xi}{n}\Big(\Vert I\Vert\log(pn)+4^{d^{\ast}}N^{d^{\ast}}2^{d^{\ast}M}\log(Cn)+\log\big(\frac{2}{\delta}\big)\Big)\bigg),\nonumber \\
 & \leq\inf_{\substack{0\leq M\leq n,\\
f\in\mathcal{F}_{d^{\ast},M}(C)
}
}\bigg(3\mathcal{E}(d^{\ast},\Theta^{\ast},f)+\frac{\Xi}{n}\Big(\Vert\Theta^{\ast}\Vert_{0}\log(pn)+4^{d^{\ast}}N^{d^{\ast}}2^{d^{\ast}M}\log(Cn)+\log\big(\frac{2}{\delta}\big)\Big)\bigg),\label{eq:preoracle}
\end{align}
with a probability of at least $1-\delta$. The rest of the proof
consists of choosing $M$ to balance the terms on the right hand side
of \ref{eq:preoracle} by using an approximation of $f^{\ast}$, namely
\begin{equation}
f_{M}=\sum_{(l,k)\in\mathcal{Z}_{M}^{d^{\ast}}}\langle f^{\ast},\Psi_{l,k}\rangle\Psi_{l,k}.\label{eq:fprojection}
\end{equation}
and then determining the projection level $M$. To do this, we have
to verify that $f_{M}$ is a valid choice for $f$ in the sense that
$f_{M}\in\mathcal{F}_{d^{\ast},M}(C)$. Indeed, the Cauchy-Schwarz
inequality ensures that
\begin{align}
L^{d^{\ast}}\sum_{(l,k)\in\mathcal{Z}_{M}^{d^{\ast}}}2^{l(d/2+1)}\vert\langle f^{\ast},\Psi_{l,k}\rangle\vert & \leq L^{d^{d}}\Big(\sum_{(l,k)\in\mathcal{Z}_{M}^{d^{\ast}}}2^{2l(1-\alpha+d^{\ast}/2)}\Big)^{\frac{1}{2}}\Big(\sum_{(l,k)\in\mathcal{Z}_{M}^{d^{\ast}}}2^{2l\alpha}\vert\langle f^{\ast},\Psi_{l,k}\rangle\vert^{2}\Big)^{\frac{1}{2}}\nonumber \\
 & \leq L^{d^{\ast}}2N^{d^{\ast}/2}8^{d^{\ast}/2}\Big(\sum_{(l,k)\in\mathcal{Z}_{\infty,\infty}^{d^{\ast}}}2^{2l\alpha}\vert\langle f^{\ast},\Psi_{l,k}\rangle\vert^{2}\Big)^{\frac{1}{2}}\nonumber \\
 & \leq C.\label{eq:leqC}
\end{align}
Using \ref{assu:nonmisspecified}, we see that $f_{M}$ admits an
excess risk of

\begin{align}
\mathcal{E}(d^{\ast},\Theta^{\ast},f_{M}) & =\E[(f_{M}(\Theta^{\ast}\mathbf{X})-f^{\ast}(\Theta^{\ast}\mathbf{X}))^{2}]\nonumber \\
 & =\int_{[-B_{1},B_{1}]^{d^{\ast}}}\varrho(x)(f_{M}(x)-f^{\ast}(x))^{2}\lebesgue^{d^{\ast}}(\d x)\nonumber \\
 & \leq B_{2}\int_{[-B_{1},B_{1}]^{d^{\ast}}}(f_{M}(x)-f^{\ast}(x))^{2}\lebesgue^{d^{\ast}}(\d x)\nonumber \\
 & \leq B_{2}\sum_{(l,k)\in\mathcal{Z}_{\infty,N}^{d^{\ast}}\setminus\mathcal{Z}_{M}^{d^{\ast}}}\vert\langle f^{\ast},\Psi_{l,k}\rangle\vert^{2}\nonumber \\
 & \leq B_{2}2^{-2\alpha M}\sum_{(l,k)\in\mathcal{Z}_{\infty,N}^{d^{\ast}}\setminus\mathcal{Z}_{M}^{d^{\ast}}}2^{2l\alpha}\vert\langle f^{\ast},\Psi_{l,k}\rangle\vert^{2}\nonumber \\
 & \leq B_{2}2^{-2\alpha M}(2N^{d^{\ast}/2}8^{d^{\ast}/2})^{-1}CL^{-d^{\ast}}.\label{eq:prerate}
\end{align}
Applying \ref{eq:prerate} to \ref{eq:preoracle}, we see that there
exists a constant $\Lambda_{1}$ only depending on $C,K,\Gamma,\sigma,N,B_{1},B_{2}$
and $d^{\ast}$ such that
\[
\mathcal{E}(\hat d_{\lambda},\hat{\Theta}_{\lambda},\hat f_{\lambda})\leq\Lambda_{1}\inf_{0\leq M\leq n}\Big(2^{-2\alpha M}+2^{d^{\ast}M}\frac{\log(Cn)}{n}+\frac{\Vert\Theta^{\ast}\Vert_{0}\log(pn)}{n}+\frac{\log\big(\frac{2}{\delta}\big)}{n}\Big)
\]
with a probability of at least $1-\delta$. To balance the order of
the terms depending on $M$, we choose
\[
M=\left\lceil \log\Big(\frac{n}{\log(Cn)}\Big)\Big/((2\alpha+d^{\ast})\log(2))\right\rceil 
\]
and obtain that for some constant $\Lambda$ depending only on $C,K,\Gamma,\sigma,N,B_{1},B_{2}$
and $d^{\ast}$, we have 
\[
\mathcal{E}(\hat d_{\lambda},\hat{\Theta}_{\lambda},\hat f_{\lambda})\leq\Lambda\bigg(\Big(\frac{\log(Cn)}{n}\Big)^{\frac{2\alpha}{2\alpha+d^{\ast}}}+\frac{\Vert\Theta^{\ast}\Vert_{0}\log(pn)}{n}+\frac{\log\big(\frac{2}{\delta}\big)}{n}\bigg)
\]
with a probability of at least $1-\delta$. This completes the proof
of \ref{cor:oracleineq}.\hfill\qed

\subsection{Proofs of auxiliary lemmata}
\begin{proof}[Proof of \ref{lem:klequal}]
 We employ another auxiliary lemma:
\begin{lem}
\label{lem:klpunishmentterm}It holds that

\begin{align*}
\mathcal{K}(\rho_{d,I,M,\eta,\gamma},\pi) & =\log(G(d,I,M))+\mathcal{K}(\rho_{d,I,M,\eta,\gamma},\pi_{d,I,M}),\qquad\text{where}\\
G(d,I,M) & \coloneqq\frac{1}{721}(1-10^{-p})(1-10^{(1-p)d-1})(10-10^{-n})10^{\Vert I\Vert+M+1}\vert\mathcal{I}_{d,\Vert I\Vert}\vert.
\end{align*}
\end{lem}

\parindent=0mmNow, we can combine
\[
\vert\mathcal{I}_{d,\Vert I\Vert}\vert=\vert\{J\mid\emptyset\neq J=J_{1}\times\cdots\times J_{d},\,J_{1},\dots,J_{d}\subseteq\{1,\dots,p\}\}\vert\leq{dp \choose \Vert I\Vert}
\]
with the basic inequality ${dp \choose \Vert I\Vert}\leq\big(\frac{dp\e}{\Vert I\Vert}\big)^{\Vert I\Vert}$
and the fact that $\Vert I\Vert\geq d$, to obtain 
\begin{align*}
\mathcal{K}(\rho_{d,I,M,\eta,\gamma},\pi) & =\log(G(d,I,M))+\mathcal{K}(\rho_{d,I,M,\eta,\gamma},\pi_{d,I,M})\\
 & \leq\log\bigg(10^{\Vert I\Vert+M+1}{dp \choose \Vert I\Vert}\bigg)+\mathcal{K}(\rho_{d,I,M,\eta,\gamma},\pi_{d,I,M})\\
 & \leq\Vert I\Vert\log(\e p)+(\Vert I\Vert+M+1)\log(10)+\mathcal{K}(\rho_{d,I,M,\eta,\gamma},\pi_{d,I,M}).\qedhere
\end{align*}

\end{proof}
\begin{proof}[Proof of \ref{lem:klequal2}]
We split the proof into two further auxiliary lemmata:
\begin{lem}
\label{lem:technicallemmatheta}For $\rho_{d,I,\eta}^{1}$ from \ref{eq:rhotheta},
we have 
\[
\mathcal{K}(\rho_{d,I,\eta}^{1},\mu_{d,I})\leq\Vert I\Vert\log\Big(\frac{3\sqrt{2}}{\eta}\Big),\,\forall\eta\in(0,1].
\]
\end{lem}

\begin{lem}
\label{lem:technicallemmaf}For $\rho_{d,M,\gamma}^{2}$ from \ref{eq:rhotheta},
we have
\[
\mathcal{K}(\rho_{d,M,\gamma}^{2},\nu_{d,M})\leq2^{3d}N^{d}2^{M+1}\log\Big(\frac{C+1}{\gamma}\Big),\,\forall\gamma\in(0,1].
\]
\end{lem}

\parindent=0mmThe assertion then follows directly via
\begin{align*}
\mathcal{K}(\rho,\pi_{d,I,M}) & =\mathcal{K}(\kappa_{d}\otimes\rho_{d,I,\eta}^{1}\otimes\rho_{d,M,\gamma}^{2},\kappa_{d}\otimes\mu_{d,I}\otimes\nu_{d,M})\\
 & =\mathcal{K}(\kappa_{d},\kappa_{d})+\mathcal{K}(\rho_{d,I,\eta}^{1},\mu_{d,I})+\mathcal{K}(\rho_{d,M,\gamma}^{2},\nu_{d,M})\\
 & \leq\Vert I\Vert\log\Big(\frac{3\sqrt{2}}{\eta}\Big)+2^{3d}N^{d}2^{M+1}\log\Big(\frac{C+1}{\gamma}\Big)\eqqcolon T_{2}.\qedhere
\end{align*}
\end{proof}
\begin{proof}[Proof of \ref{lem:klpunishmentterm}]
 To simplify the notation we write $\rho=\rho_{d,I,M,\eta,\gamma}$
and $\pi_{d,I,M}=\kappa_{d,}\otimes\mu_{d,I}\otimes\nu_{d,M}$. We
will show that
\begin{equation}
\diff{\rho}{\pi}=G(d,I,M)\diff{\rho}{\pi_{d,I,M}}\label{eq:rhodensity}
\end{equation}
from which we can deduce
\[
\mathcal{K}(\rho,\pi)=\int\log\Big(\diff{\rho}{\pi}\Big)\d\rho=\log(G(d,I,M))+\int\log\Big(\diff{\rho}{\pi_{d,I,M}}\Big)\d\rho=\log(G(d,I,M))+\mathcal{K}(\rho,\pi_{d,I,M}).
\]
For \ref{eq:rhodensity}, we need to show that
\begin{equation}
\rho(A)=\int_{A}G(d,I,M)^{-1}\diff{\rho}{\pi}\d\pi_{d,I,M}\label{eq:rhodensity2}
\end{equation}
holds for all $A=A_{1}\times A_{2}\times A_{3}$ with $A_{1}\in2^{\{1,\dots,p\}}$,
$A_{2}\in\mathscr{B}_{\mathcal{S}}$ and $A_{3}\in\mathscr{B}_{\mathcal{F}(C+1)}$.Observe
that for the sets
\begin{align*}
\mathcal{S}_{d,\Leftrightarrow}(J) & \coloneqq\{\Theta=(\vartheta_{1},\dots,\vartheta_{d})^{\top}\in\mathcal{S}_{d}\mid(\vartheta_{i,j}\neq0\Leftrightarrow j\in J_{i})\,\forall i\in\{1,\dots,d\},j\in\{1,\dots,p\}\}\\
\mathcal{F}_{d,\tilde M,\neq}(C+1) & \coloneqq\Big\{ f=\Phi_{d,\tilde M}(\beta)\in\mathcal{F}_{d,\tilde M}(C+1)\,\\
 & \qquad\qquad\Big|\,\exists(l,(k_{1},k_{2}))\in\mathcal{Z}_{\tilde M}^{d}:((\Vert k_{1}\Vert_{\infty}=2^{l}\tilde M)\land(\beta_{l,k}\neq0))\land(k_{2}\neq0)\Big\}
\end{align*}
with $J=J_{1}\times\cdots\times J_{d}\in\mathcal{I}_{d}$ and $\tilde M\in\{0,\dots,n\}$,
we have
\begin{equation}
\kappa_{d}(\{d\}\}=\mu_{d,J}(\mathcal{S}_{d,\Leftrightarrow}(J))=\nu_{d,\tilde M}(\mathcal{F}_{d,\tilde M,\neq}(C+1))=1.\label{eq:eqneq}
\end{equation}
In particular, \ref{eq:eqneq} holds for $J=I$ and $\tilde M=M$.
Since also $\rho(\{d\}\times\mathcal{S}_{d,\Leftrightarrow}(I)\times\mathcal{F}_{d,M,\neq}(C+1))=1$,
no generality is lost in additionally assuming that
\[
A_{1}=\{d\},\qquad A_{2}\subseteq\mathcal{S}_{d,\Leftrightarrow}(I)\qquad\text{and}\qquad A_{3}\subseteq\mathcal{F}_{d,M,\neq}(C+1).
\]
Now note that 
\begin{equation}
\{c\}\cap\{d\}=\emptyset\,\forall c\neq d,\quad\mathcal{S}_{d,\Leftrightarrow}(J)\cap\mathcal{S}_{d,\Leftrightarrow}(I)=\emptyset\,\forall J\neq I\quad\text{and}\quad\mathcal{F}_{d,\tilde M,\neq}(C+1)\cap\mathcal{F}_{d,M,\neq}(C+1)=\emptyset\,\forall\tilde M\neq M.\label{eq:emptysets}
\end{equation}
Combining \ref{eq:eqneq} with \ref{eq:emptysets}, we see that
\[
\int_{A_{2}}\diff{\rho}{\pi}\d\mu_{d,J}=0\forall J\neq I\qquad\text{and}\qquad\int_{A_{3}}\diff{\rho}{\pi}\d\nu_{d,M}=0\forall\tilde M\ne M.
\]
Therefore, repeated application of Fubini's theorem yields
\begin{align*}
\rho(A) & =\int_{A}\diff{\rho}{\pi}\d\pi\\
 & =\sum_{c=1}^{p}10^{-c}\int_{A}\diff{\rho}{\pi}\d\pi_{c}\Big/\Big(\frac{1}{\text{9}}(1-10^{-p})\Big)\\
 & =\int_{A}\diff{\rho}{\pi}\d\kappa_{d}\otimes\mu_{d}\otimes\nu_{d}\Big/\Big(\frac{10^{d}}{\text{9}}(1-10^{-p})\Big)\\
 & =\int_{A_{2}\times A_{3}}\diff{\rho}{\pi}(\Theta,f)\d\mu_{d}\otimes\nu_{d}(\Theta,f)\Big/\Big(\frac{10^{d}}{\text{9}}(1-10^{-p})\Big)\\
 & =G(d,I,M)^{-1}\int_{A_{2}\times A_{3}}\diff{\rho}{\pi}(\Theta,f)\d\mu_{d}\otimes\nu_{d}(\Theta,f)\\
 & =G(d,I,M)^{-1}\int_{A}\diff{\rho}{\pi}\d\pi_{d,I,M}.
\end{align*}
Thus, we have shown \ref{eq:rhodensity2}.
\end{proof}
\begin{proof}[Proof of \ref{lem:technicallemmatheta}]
We will show that 
\begin{equation}
\mathcal{K}(\rho_{d,I,\eta}^{1,i},\mu_{I_{i}})\leq\vert I_{i}\vert\log\Big(\frac{3\sqrt{2}}{\eta}\Big),\,\forall\eta\in(0,1],\,i=1,\dots,d.\label{eq:mainwork}
\end{equation}
 The assertion follows immediately via
\[
\mathcal{K}(\rho,\mu_{d,I})=\mathcal{K}\Big(\bigotimes_{i=1}^{d}\rho_{d,I,\eta}^{1,i},\bigotimes_{i=1}^{d}\mu_{I_{i}}\Big)=\sum_{i=1}^{d}\mathcal{K}(\rho_{d,I,\eta}^{1,i},\mu_{I_{i}})\leq\sum_{i=1}^{d}\vert I_{i}\vert\log\Big(\frac{3\sqrt{2}}{\eta}\Big)=\Vert I\Vert\log\Big(\frac{3\sqrt{2}}{\eta}\Big),
\]
where the first equality holds barring a slight breach of conventions
for product measures. To show \ref{eq:mainwork}, fix $i\in\{1,\dots,d\}$
and for simplicity of the notation, set $\rho\coloneqq\rho_{d,I,\eta}^{1,i}$
and $J=I_{i}$. Plugging the $\mu_{J}$-density of $\rho$ into the
definition of the Kullback-Leibler divergence, we easily obtain 
\begin{equation}
\mathcal{K}(\rho,\mu_{J})=-\log\Big(\int\1_{\{\Vert\vartheta-\vartheta_{d,I,i}^{\ast}\Vert_{2}\leq\eta\}}\mu_{J}(\d\vartheta)\Big)=-\log(\tilde{\mu}(\{\vartheta\in\R^{\vert J\vert}\mid\Vert\vartheta-\tilde{\vartheta}^{\ast}\Vert_{2}\leq\eta\}))\label{eq:proplowerbound}
\end{equation}
 where $\tilde{\vartheta}^{\ast}$ is the projection of $\vartheta_{d,I,i}^{\ast}$
onto the coordinates whose indices are elements of $J$ and $\tilde{\mu}_{J}$
denotes the uniform distribution on the unit sphere in $\R^{\vert J\vert}$.

We want to show a lower bound for $\tilde{\mu}(\{\vartheta\in\R^{\vert J\vert}\mid\Vert\vartheta-\tilde{\vartheta}^{\ast}\Vert_{2}\leq\eta\})$,
which is the proportion of the surface of the unit sphere in $\R^{\vert J\vert}$
which is covered by the $\eta$-ball around $\tilde{\vartheta}^{\ast}$
to the surface of the entire unit sphere. By rotational symmetry of
the uniform distribution on this sphere, no generality is lost by
assuming $\tilde{\vartheta}^{\ast}=(1,0,\dots,0)^{\top}\in\R^{\vert J\vert}$.

For any $\tilde{\vartheta}=(\tilde{\vartheta}_{1},\dots,\vartheta_{\vert J\vert})^{\top}\neq0$
with $\Vert\tilde{\vartheta}\Vert_{2}\leq1$, the $\eta$-ball around
$\tilde{\vartheta}^{\ast}$ covers the same part of the surface of
the unit sphere as the $\tilde{\eta}$-ball around $\tilde{\vartheta}$,
where $\tilde{\eta}=\sqrt{\eta^{2}\Vert\tilde{\vartheta}\Vert_{2}-2\Vert\tilde{\vartheta}\Vert_{2}+\Vert\tilde{\vartheta}\Vert_{2}^{2}+1}$.
Using this dependence between $\tilde{\vartheta},\eta$ and $\tilde{\eta}$,
it is easily checked that $\Vert\tilde{\vartheta}\Vert_{2}=\sqrt{1-\tilde{\eta}^{2}}$
is solved for $0<\tilde{\eta}=\sqrt{\eta^{2}-\frac{1}{4}\eta^{4}}<1$.
Henceforth, fix this $\tilde{\eta}$. Suppose that $N_{\tilde{\eta}}$
is the smallest number of $\tilde{\eta}$-balls with centers in the
unit ball that suffice to cover the entire unit ball. Denote their
centers by $t_{1},\dots,t_{N_{\tilde{\eta}}}$. In particular, these
balls cover the entire unit sphere and therefore, at least one of
them covers at least $N_{\tilde{\eta}}^{-1}$ of the surface of the
unit sphere. This can only be the case for $t_{j}\neq0$, because
otherwise $\tilde{\eta}<1$ implies $\{\vartheta\in\R^{\vert J\vert}\mid\Vert\vartheta-t_{j}\Vert_{2}\leq\tilde{\eta}\}\cap\{\vartheta\in\R^{\vert J\vert}\mid\Vert\vartheta\Vert_{2}=1\}=\emptyset$.
If we now change the length of such a $\vartheta\coloneqq t_{j}\neq0$
(without changing its orientation and without changing $\tilde{\eta}$),
the coverage of the unit sphere provided by the corresponding $\tilde{\eta}$-ball
also changes. In particular, we will show that if $\Vert\vartheta\Vert_{2}\neq\sqrt{1-\tilde{\eta}^{2}}$,
then decreasing (or increasing) $\Vert\vartheta\Vert_{2}$ towards
$\sqrt{1-\tilde{\eta}^{2}}$, enlarges the coverage of the corresponding
ball on the unit sphere.  Thus, the proportional coverage of the
$\tilde{\eta}$-ball around $\Vert\vartheta\Vert_{2}^{-1}\sqrt{1-\tilde{\eta}^{2}}\vartheta$
(which has, as we showed above, the same proportional coverage of
the unit sphere as the $\eta$-ball around $\tilde{\vartheta}^{\ast}$
that we are actually trying to control) is bounded from below by the
proportional coverage of the $\tilde{\eta}$-ball around $\tilde{\vartheta}$,
which in turn is bounded from below by $N_{\tilde{\eta}}^{-1}$. Using
the fact that $N_{\tilde{\eta}}\leq(\frac{3}{\tilde{\eta}})^{\vert J\vert}$
combined with $\tilde{\eta}\geq\eta/\sqrt{2}$, we have 
\[
N_{\tilde{\eta}}^{-1}\geq\Big(\frac{3}{\tilde{\eta}}\Big)^{-\vert J\vert}\geq\Big(\frac{3\sqrt{2}}{\eta}\Big)^{-\vert J\vert}
\]
and therefore \ref{eq:mainwork} follows from \ref{eq:proplowerbound}
with 
\[
\mathcal{K}(\rho,\mu_{J})=-\log(\tilde{\mu}(\{\vartheta\in\R^{\vert J\vert}\mid\Vert\vartheta-\tilde{\vartheta}^{\ast}\Vert_{2}\leq\eta\}))\leq-\log(N_{\tilde{\eta}}^{-1})\leq\vert J\vert\log\Big(\frac{3\sqrt{2}}{\eta}\Big).
\]
It remains to show that changing the length of $\vartheta\neq0$ towards
$\sqrt{1-\tilde{\eta}^{2}}$ increases the proportional coverage of
the corresponding $\tilde{\eta}$-ball. By rotational symmetry, we
can assume $\vartheta=(\vartheta_{1},0,\dots,0)^{\top}\in\R^{\vert J\vert}$
with some $0<\vartheta_{1}\leq1$ at no loss of generality. Now, it
is sufficient to show that 
\[
\{y\in\R^{\vert J\vert}\mid(\Vert y-\vartheta\Vert_{2}\leq\tilde{\eta})\land(\Vert y\Vert_{2}=1)\}\subseteq\{y\in\R^{\vert J\vert}\mid(\Vert y-\tilde{\vartheta}\Vert_{2}\leq\tilde{\eta})\land(\Vert y\Vert_{2}=1)\}
\]
where $\tilde{\vartheta}=(\sqrt{1-\tilde{\eta}^{2}},0,\dots,0)^{\top}\in\R^{\vert J\vert}$.
In this setting, and as $\eta,\tilde{\eta}>0$, the relationship 
\begin{equation}
\tilde{\eta}^{2}=\eta^{2}\Vert\tilde{\vartheta}\Vert_{2}-2\Vert\tilde{\vartheta}\Vert_{2}+\Vert\tilde{\vartheta}\Vert_{2}^{2}+1\label{eq:tildeta}
\end{equation}
is equivalent to 
\begin{equation}
\eta^{2}=\frac{2\tilde{\vartheta}_{1}-\tilde{\vartheta}_{1}^{2}-1+\tilde{\eta}^{2}}{\tilde{\vartheta}_{1}}.\label{eq:eta}
\end{equation}
Using elementary calculus techniques together with the fact that
$\tilde{\eta}<1$, it is easy to see that 
\[
\frac{2\vartheta_{1}-\vartheta_{1}^{2}-1+\tilde{\eta}^{2}}{\vartheta_{1}}\leq2(1-\sqrt{1-\tilde{\eta}^{2}}).
\]
Combining this with the relationship between $\eta$ and $\tilde{\eta}$,
we obtain 
\begin{align*}
\{y\in\R^{\vert J\vert}\mid(\Vert y-\vartheta\Vert_{2}\leq\tilde{\eta})\land(\Vert y\Vert_{2}=1)\} & =\{y\in\R^{\vert J\vert}\mid(\Vert y-\tilde{\vartheta}^{\ast}\Vert_{2}^{2}\leq\frac{2\vartheta_{1}-\vartheta_{1}^{2}-1+\tilde{\eta}^{2}}{\vartheta_{1}})\land(\Vert y\Vert_{2}=1)\}\\
 & \subseteq\{y\in\R^{\vert J\vert}\mid(\Vert y-\tilde{\vartheta}^{\ast}\Vert_{2}^{2}\leq2(1-\sqrt{1-\tilde{\eta}^{2}}))\land(\Vert y\Vert_{2}=1)\}\\
 & =\{y\in\R^{\vert J\vert}\mid(\Vert y-\tilde{\vartheta}^{\ast}\Vert_{2}^{2}\leq\eta^{2})\land(\Vert y\Vert_{2}=1)\}\\
 & =\{y\in\R^{\vert J\vert}\mid(\Vert y-\tilde{\vartheta}\Vert_{2}\leq\tilde{\eta})\land(\Vert y\Vert_{2}=1)\}.\qedhere
\end{align*}
\end{proof}
\begin{proof}[Proof of \ref{lem:technicallemmaf}]
 To simplify the notation, we write $\tilde f^{\ast}=f_{d,M}^{\ast}$
and $\tilde{\beta}^{\ast}=\beta_{d,M}^{\ast}$. We will show that
\begin{equation}
\int\1_{\{\Vert f-\tilde f^{\ast}\Vert_{\psi}\leq\gamma\}}\d\nu_{d,M}(f)=\Big(\frac{C+1}{\gamma}\Big)^{-\vert\mathcal{Z}_{M}^{d}\vert}.\label{eq:mainwork2}
\end{equation}
The assertion follows directly with

\[
\mathcal{K}(\rho_{d,M,\gamma}^{2},\nu_{d,M})=-\log\Big(\int\1_{\{\Vert f-f^{\ast}\Vert_{\psi}\leq\gamma\}}\d\nu_{d,M}(f)\Big)=\vert\mathcal{Z}_{M}^{d}\vert\log\Big(\frac{C+1}{\gamma}\Big)\leq4^{d}N^{d}2^{dM+1}\log\Big(\frac{C+1}{\gamma}\Big).
\]
We now show \ref{eq:mainwork2} using the definition of $\nu_{d,M}$.
If we let $\lebesgue^{\mathcal{Z}_{M}^{d}}$ denote the Lebesgue measure
on $\mathcal{B}_{d,M}(C+1)$, we obtain 
\begin{align*}
\int\1_{\{\Vert f-\tilde f^{\ast}\Vert_{\psi}\leq\gamma\}}\d\nu_{d,M}(f) & =\int\1_{\{\Vert\Phi_{d,M}(\beta)-\tilde f^{\ast}\Vert_{\psi}\leq\gamma\}}\d\tilde{\nu}_{d,M}(\beta)\\
 & =\int\1_{\{\Vert\beta-\tilde{\beta}^{\ast}\Vert_{\mathcal{B}}\leq\gamma\}}\d\tilde{\nu}_{d,M}(\beta)\\
 & =\frac{\int\1_{\{\Vert\beta\Vert_{\mathcal{B}}\leq C+1\}}\1_{\{\Vert\beta-\tilde{\beta}^{\ast}\Vert_{\mathcal{B}}\leq\gamma\}}\d\lebesgue^{\mathcal{Z}_{M}^{d}}(\beta)}{\int\1_{\{\Vert\beta-\tilde{\beta}^{\ast}\Vert_{\mathcal{B}}\leq C+1\}}\d\lebesgue^{\mathcal{Z}_{M}^{d}}(\beta)}\\
 & =\frac{\int\1_{\{\Vert\beta-\tilde{\beta}^{\ast}\Vert_{\mathcal{B}}\leq\gamma\}}\d\lebesgue^{\mathcal{Z}_{M}^{d}}(\beta)}{\int\1_{\{\Vert\beta-\tilde{\beta}^{\ast}\Vert_{\mathcal{B}}\leq C+1\}}\d\lebesgue^{\mathcal{Z}_{M}^{d}}(\beta)}\\
 & =\Big(\frac{\gamma}{C+1}\Big)^{\vert\mathcal{Z}_{M}^{d}\vert}\frac{\int\1_{\{\Vert\beta-\tilde{\beta}^{\ast}\Vert_{\mathcal{B}}\leq1\}}\d\lebesgue^{\mathcal{Z}_{M}^{d}}(\beta)}{\int\1_{\{\Vert\beta-\tilde{\beta}^{\ast}\Vert_{\mathcal{B}}\leq1\}}\d\lebesgue^{\mathcal{Z}_{M}^{d}}(\beta)}\\
 & =\Big(\frac{C+1}{\gamma}\Big)^{-\vert\mathcal{Z}_{M}^{d}\vert},
\end{align*}
where we have used that 
\[
\Vert\beta\Vert_{\mathcal{B}}\leq\Vert\tilde{\beta}^{\ast}\Vert_{\mathcal{B}}+\Vert\beta-\tilde{\beta}^{\ast}\Vert_{\mathcal{B}}\leq C+\gamma\leq C+1
\]
on $\{\Vert\beta-\tilde{\beta}^{\ast}\Vert_{\mathcal{B}}\leq\gamma\}$
in the fourth equality. This implies that if the first indicator in
the integral in the numerator is $1$, so is the second.
\end{proof}
\bibliographystyle{apalike2}
\bibliography{sources}

\end{document}